\documentclass{amsart}

\usepackage{graphicx}
\usepackage{enumerate}
\newtheorem{theorem}{Theorem}
\newtheorem{lemma}[theorem]{Lemma}

\theoremstyle{definition}
\newtheorem{definition}[theorem]{Definition}

\newtheorem{cor}[theorem]{Corollary}

\theoremstyle{remark}
\newtheorem{remark}[theorem]{Remark}

\numberwithin{equation}{section}

\newcommand{\Hi}{\mathbb{H}}
\newcommand{\Sf}{\mathbb{S}}
\newcommand{\R}{\mathbb{R}}
\newcommand{\Lo}{\mathbb{L}}
\newcommand{\E}{\mathbb{E}}
\newcommand{\Q}{Q^{n}(\varepsilon)}

\begin{document}

\title{Intrinsic and extrinsic geometry of hypersurfaces in $\mathbb{S}^n \times \mathbb{R}$ and $\mathbb{H}^n \times \mathbb{R}$}

\author{Rafael Novais}
\address{Departamento de Matem\'atica, Universidade de
Bras\'ilia, 70910-900, Bras\'ilia-DF, Brazil}
\email{r.m.novais@mat.unb.br}
\thanks{The first author was supported by CNPq}

\author{Jo\~ao Paulo dos Santos}
\address{Departamento de Matem\'atica, Universidade de
Bras\'ilia, 70910-900, Bras\'ilia-DF, Brazil}
\email{j.p.santos@mat.unb.br}
\thanks{The second author was supported by FAPDF 0193.001346/2016}

\subjclass[2010]{53B25; 53C42}

%

\keywords{hypersurfaces in product spaces, conformally flat, radially flat, rotation, semi-parallel}

\begin{abstract}
In this paper, geometric characterizations of conformally flat and radially flat hypersurfaces in $\mathbb{S}^n \times \mathbb{R}$ and $\mathbb{H}^n \times \mathbb{R}$ are given by means of their extrinsic geometry. Under suitable conditions on the shape operator, we classify conformally flat hypersurfaces in terms of rotation hypersurfaces. In addition, a close relation between radially flat hypersurfaces and semi-parallel hypersurfaces is established. These results lead to geometric descriptions of hypersurfaces with special intrinsic structures, such as Einstein metrics, Ricci solitons and hypersurfaces with constant scalar curvature. 

\end{abstract}

\maketitle

\section{Introduction}

Hypersurfaces in the product spaces $\Hi^n \times \R$ and $\Sf^n \times \R$ have attracted a lot of attention in recent years. Properties regarding their intrinsic geometry and the relations with the ambient space have been considered mainly in the context of constant sectional curvature. In \cite{aledo1} and \cite{aledo2} the two dimensional case is considered and surfaces with constant Gaussian curvature are classified. For higher dimensions, hypersurfaces with constant secctional curvature were considered in \cite{manfio}, where the authors have proved that those hypersurfaces must be an open part of a rotation hypersurface. Such invariant property were introduced in \cite{Veken3} and since then, they have been playing an important role in the classification of hypersurfaces in $\Hi^n \times \R$ and $\Sf^2 \times \R$ with special geometric aspects, such as totally umbilicity, parallelism and semi-parallelism, as we can see in \cite{Veken1} and \cite{Veken2}.

A Riemannian manifold is conformally flat if each point has a neighborhood where the metric is conformal to a flat metric, i.e., a metric with zero sectional curvature. The investigation of conformally flat hypersurfaces in Riemannian manifolds, equipped with the induced metric, has been of interest for some time (see \cite{lafontaine, jeromin, joao} and the references therein). Since the problem of their classification proposed by Cartan in \cite{cartan}, the relationship between the intrinsic and extrinsic geometry has been considered by taking into account the geometry of the ambient space. For instance, when the ambient manifold is also conformally flat, Nishikawa and Maeda \cite{Nishikawa} have proved that $n$-dimensional conformally flat hypersurfaces must be \emph{quasi-umbilical}, i.e., one of the the principal curvatures has multiplicity at least $(n-1)$. In our case, we will see that rotation hypersurfaces are conformally flat. Conversely, conformally flat hypersurfaces, with additional conditions on the shape operator, are given by rotation hypersurfaces (Theorem \ref{conformally}). As applications, we will use this result to deal with important Riemannian structures on hypersurfaces in $\Hi^n \times \R$ and $\Sf^2 \times \R$, namely, the Einstein metrics and the Ricci solitons. We also characterize the conformally flat hypersurfaces with constant scalar curvature (see Section 5).

On the other hand, radially flat Riemannian manifolds are the manifolds endowed with a smooth vector field $X$ where the sectional curvatures vanish along planes that contained the vector field $X$. Radially flat Riemannian manifolds constitute an important class of metrics and were considered, for example, in the context of Ricci solitons \cite{petersen1, garciario, petersen2}. In this case, the vector field considered is the potential vector field of the soliton. It turns out that the radially flat condition can be seen, in some sense, as a weakening of the flatness condition and, consequently, more information about such metrics can be obtained. This situation will be seen in our context as a generalization of a result given in \cite{Veken3} for intrinsically flat rotation hypersurfaces in $\Sf^n \times \R$ and $\Hi^n \times \R$. Our main result regarding radially flat hypersurfaces is a close relation between the  geometry of radially flat hypersurfaces and the geometry of semi-parallel hypersurfaces in such spaces (Theorem \ref{radially}).

The paper is organized as follows: in Section 2, the main results presented. Section 3 is a brief section with preliminaries for the study of hypersurfaces in $Q^n(\varepsilon) \times\mathbb{R}$. The tools and basic results that will be used in the proofs of the main results will be given in this section. In Section 4, the proofs of the main results are given. Finally, in Section 5, we give some interesting applications as corollaries of the main results. 

\section{Statement of the main results}

In order to state our results, let us first establish some notation. Let $Q^n(\varepsilon)$ be the unit sphere $\Sf^n$, if $\varepsilon=1$, or the hyperbolic space and $\Hi^n$ if $\varepsilon=-1$ and consider the manifold  $Q^{n}(\varepsilon)\times\mathbb{R}$ given by:
$$
\begin{array}{rcl}
        {\Sf}^n \times \R &=& \{(x_1,\ldots,x_{n+2})\in\E^{n+2}|\;x_1^2+x_2^2+\ldots+x_{n+1}^2=1\}, \\
        {\Hi}^n \times \R &=& \{(x_1,\ldots,x_{n+2})\in\Lo^{n+2}|-x_1^2+x_2^2+\ldots+x_{n+1}^2=-1, x_1>0\},
\end{array}
$$
with the metric induced by the ambient space, where $\E^{n+2}$ is the $(n+2)-$dimensional Euclidean space and $\Lo^{n+2}$ is the $(n+2)-$dimensional Lorentzian space with the canonical metric $ds^2=-dx_1^2+dx_2^2+ \ldots +dx_{n+2}^2$.

Let $M^n$ be a hypersurface in $Q^{n}(\varepsilon)\times\mathbb{R}$  with unit normal $N$ and let $\partial_{x_{n+2}}$ be the coordinate vector field of the second factor $\R$. The orthogonal projection of  $\partial_{x_{n+2}}$ over the tangent space of  $M^n$ will be denoted by $T$. Also, let $\theta$ be the angle function between $N$ and  $\partial_{x_{n+2}}$. Then we have the following decomposition
$$\partial_{x_{n+2}}=T+\cos{\theta}N.$$

In this context, we have our first main result: 

\begin{theorem} \label{conformally}
Let $M^n$, $n>3$, be a hypersurface in $Q^n(\varepsilon) \times \mathbb{R}$. If $M^n$ is a rotation hypersuface, then $M^n$ is conformally flat. Conversely, if $M^n$ is a conformally flat hypersurface, then either $M^n$ is a totally umbilical hypersurface or its shape operator has two distinct eigenvalues of multiplicity $n-1$ and $1$. In this case, $M^n$ is locally congruent to a rotation hypersurface when one of following cases occurs:
\begin{enumerate}[i)]
\item $M^n$ is a totally umbilical hypersurface, which is not totally geodesic;
\item the shape operator of $M^n$ has two distinct eingevalues $\lambda$ and $\mu$, of multiplicity $1$ and $n-1$, respectively, where $\lambda=\lambda(\mu, \theta)$ and the vector field  $T$ is a principal direction.
\end{enumerate}
\end{theorem}

It is important to note that the hypothesis of the vector field $T$ being an eigenvalue of the shape operator is not very restrictive, nor artificial. In fact, hypersurfaces in $Q^n(\varepsilon) \times \R$ with such feature were considered by Tojeiro in \cite{tojeiro} (see Section 3). As pointed out by the author, this class of hypersurfaces includes all rotation hypersurfaces, all hypersurfaces with constant sectional curvature, with dimension $ n\geq 3$ \cite{manfio}, and also all constant angle hypersurfaces, that is, hypersurfaces with the property that its unit normal vector field makes a constant angle with the unit vector field $\partial_{x_{n+2}}$ \cite{manfio, tojeiro, dillen2, dillen3}. Besides that, Tojeiro has shown that $T$ is a principal direction for a hypersurface $M^n$ in $Q^n(\varepsilon) \times \R$ if, and only if, $M^n$ has flat normal bundle as a submanifold into $\E^{n+2}$, resp. $\Lo^{n+2}$. In general, a hypersurface $M^n \subset \tilde{M}^{n+1}$ is said to have a canonical principal direction relative
to a vector field $X$ in $M^n \subset \tilde{M}^{n+1}$  if the projection of $X$ over the tangent space of the hypersurface is an eigenvector of the shape operator. Hypersurfaces with such property was considered in \cite{garnica} and also in \cite{Veken4, dillen}, in the context of surfaces in $Q^n(\varepsilon) \times \R$.

The second part of Theorem \ref{conformally}, item ii), is a result similar to the classification of $n$-dimensional conformally flat hypersurfaces in space forms, $n>3$, obtained by do Carmo and Dajczer in \cite{carmo}. The authors proved that, if the shape operator of the hypersurface has two distinct eigenvalues $\lambda$ and $\mu$, with multiplicity $1$ and $n-1$, with $\lambda \neq 0$ and $\mu=\mu(\lambda)$, then the hypersurface is contained in a rotation hypersurface \cite[Corollary 4.2]{carmo}.

Now we turn our attention to radially flat hypersurfaces in $Q^n(\varepsilon) \times \R$. Although Dillen, Fastenakels and van der Veken have shown in \cite{Veken3} that there is no intrinsically  flat rotation hypersurfaces in $Q^n(\varepsilon) \times \mathbb{R}$, when $n \geq 3$, it is possible to weaken such a hypothesis to obtain rotation hypersurfaces with an interesting geometric property. Therefore, instead of flat metrics, we will ask for zero sectional curvature along specific planes. A hypersurface $M^n$ in $Q^n(\varepsilon) \times \mathbb{R}$ will be called radially flat if the sectional curvatures along planes containing the vector field $T$ vanish, i.e., $K_M(T,X)=0$, for any vector field $X$. 

We will show a close relation between radially flat hypersurfaces and semi-parallel hypersurfaces in $Q^n(\varepsilon) \times \mathbb{R}$, i.e., hypersurfaces where the second fundamental form $h$ and the curvature tensor $R$ satisfy
$
h(R(X,Y)Z,W) + h(R(X,Y)W,Z)=0,
$
for every $X, \, Y, \, Z, \, W$ arbitrary vector fields tangent to $M^n$. Our result will provide an important intrinsic characterization for such hypersurfaces that were classified in  \cite{Veken1} and \cite{Veken2}. 

Firstly, we remark that radially flat surfaces in $Q^2(\varepsilon) \times \mathbb{R}$ are the flat surfaces and a surface in $Q^2(\varepsilon) \times \mathbb{R}$ is flat if, and only if, is semi-parallel. In fact, it follows directly from the definition of semi-parallel hypersurfaces that every flat surface should be semi-parallel. Conversely, as we can see in \cite{Veken1, Veken2}, every semi-parallel surface is flat. This observation is generalized for higher dimensions by the following theorem:

\begin{theorem} \label{radially}
Let $M^n$, $n>3$, be a hypersurface in $Q^n(\varepsilon) \times \mathbb{R}$. If $M^n$ is radially flat and $T$ is a principal direction, for a principal curvature $\lambda \neq 0$, then $M^n$ is a semi-parallel, rotation hypersurface. Conversely, if $M^n$ is a semi-parallel, not totally umbilical hypersurface, then $M^n$ is radially flat.
\end{theorem}	

\section{Hypersurfaces in $\mathbb{S}^n\times\mathbb{R}$ and $\mathbb{H}^n\times\mathbb{R}$}


Let $\nabla$, $R$, $S$ be the Riemannian connection, the curvature tensor and the shape operator of  a hypersurface $M^n$ in $Q^n(\varepsilon) \times \R$, respectively. The Gauss and Codazzi equations are given by

    \begin{multline}\label{gauss}
        \langle R(X,Y)Z,W\rangle=\varepsilon(\langle X,W\rangle\langle Y,Z\rangle-\langle X,Z\rangle\langle Y,W\rangle\\
       + \langle X,T\rangle\langle Z,T\rangle\langle Y,W\rangle+\langle Y,T\rangle\langle W,T\rangle\langle X,Z\rangle\\
        -\langle Y,T\rangle\langle Z,T\rangle\langle X,W\rangle-\langle X,T\rangle\langle W,T\rangle\langle Y,Z\rangle)\\
        +\langle SX,W\rangle\langle SY,Z\rangle-\langle SX,Z\rangle\langle SY,W\rangle.
    \end{multline}
    
    \begin{equation}\label{codazzi}
         \nabla_{X}(SY)-\nabla_{Y}(SX)-S[X,Y]=\varepsilon\cos\theta[\langle Y,T\rangle X-\langle X,T\rangle Y].
    \end{equation}   
    
Since the vector field $\partial_{x_{n+2}}$ is parallel in $\Q\times\R$, one has
    \begin{equation}\label{Xcos}
    \begin{array}{rcl}
        \nabla_{X}T&=&\cos(\theta) SX, \\
           X[\cos\theta]&=&-\langle X,ST\rangle.
     \end{array}      
    \end{equation}
    
Let $h$ be the second fundamental form of $M^n$ given by $h(X,Y) = \langle SX, Y \rangle$. If $h(X,Y) = \lambda \langle X, Y \rangle$, for some smooth function $\lambda$ defined on $M^n$, $M^n$ is called \emph{totally umbilical} and $M^{n}$ is called \emph{totally geodesic} if $h \equiv 0$. Besides that, $M^n$ is called \emph{semi-parallel} if $R \cdot h = 0$, where 
\begin{equation}
(R \cdot h) (X, \,  Y, \,  Z, \,  W) = - h(R(X, \, Y)Z, \, W) - h(R(X, \, Y)W, \, Z). \label{eq-semi-parallel}
\end{equation}
        
In \cite{Veken3}, the definition of rotation hypersurfaces in $Q^n(\varepsilon) \times \mathbb{R}$ was introduced as the following:

\begin{definition}[\cite{Veken3}]
Consider a three-dimensional subspace $P^3$ of $\E^{n+2}$ resp. $\Lo^{n+2}$, containing the $x_{n+2}-$axis. Then $(Q(\varepsilon)^n \times \R) \cap P^3 = Q^1(\varepsilon) \times \R$. Let $P^2$ be a two-dimensional subspace of $P^3$, also through the $x_{n+2}-$axis. Denote by $I$ the group of isometries of $\E^{n+2}$, resp. $\Lo^{n+2}$, which leave $Q(\varepsilon)^n \times \R$ globally invariant and which leave $P^2$
pointwise fixed. Finally, let $\alpha$ be a curve in $Q(\varepsilon)^1 \times \R$ which does not intersect $P^2$. The rotation hypersurface $M^n$ in $Q(\varepsilon)^n \times \R$ with profile curve
$\alpha$ and axis $P^2$ is defined as the $I-$orbit of $\alpha$. \label{defi-rotational}
\end{definition}

In the same paper, the authors have given a complete description of the shape operator of rotation hypersurfaces. From the Definition \ref{defi-rotational}, they obtained local parametrizations for such hypersurfaces and calculated the principal curvatures, concluding that the shape operator has just two eingenvalues, one of then with multiplicity at least $n-1$ and also that the vector field $T$ is an eigenvector of the shape operator. Besides that, they have shown an important criterium to a hypersurface in $Q^n(\varepsilon) \times \mathbb{R}$ to be a rotation hypersurface:     
    \begin{theorem}[\cite{Veken3}]
Take $n\geq3$ and let $M^{n}$ be a hypersurface in $Q^{n}(\varepsilon)\times\mathbb{R}$  with shape operator 
        $$S=\left(\begin{array}{cccc}
        \lambda & \quad  & \quad  & \quad\\
        \quad   & \mu    & \quad  & \quad\\
        \quad   & \quad  & \ddots & \quad\\
        \quad   & \quad  & \quad  & \mu
        \end{array}\right),$$
with $\lambda\neq\mu$ and suppose that $ST=\lambda T$. Assume moreover that there is a functional relation $\lambda=\lambda(\mu)$. Then $M^{n}$ is a open part of a rotation hypersurface. \label{criterion}
\end{theorem}

\begin{remark}
As pointed out by the authors in \cite[proof of Theorem 5]{Veken1} and \cite[proof of Theorem 4.2]{Veken2}, in order to guarantee that a hypersurface with shape operator given in Theorem \ref{criterion} is actually a rotation hypersurface, it is enough a functional relation $\lambda = \lambda (\mu, \theta)$ as long as the angle function $\theta$ does not vary in directions orthogonal to $T$. This fact will be used in the proof of our results. \label{rem-criterion}
\end{remark}

Let us note that rotation hypersurfaces constitute a class of hypersurfaces where the vector field $T$ is a principal direction. In \cite{tojeiro}, Tojeiro has shown a classification of the hypersurfaces for which $T$ is an eigenvector of the shape operator: 

Let $M^{n-1}$ be a hypersurface $Q^n(\varepsilon)$  and let $g_s : M^{n-1} \rightarrow Q^n (\varepsilon)$ be its family of parallel hypersurfaces, given by
$$
g_s(x) = C_{\varepsilon} (s) g(x) + S_{\varepsilon} (s) N(x),
$$
where $N$ is a unit normal vector field to $g$ and the functions $C_{\varepsilon}$ and $S_{\varepsilon}$ are given by
$$
C_{\varepsilon}(s) = \left\{ 
\begin{array}{l}
\cos(s), \, if \varepsilon = 1, \\
\cosh(s), \, if \varepsilon = -1,
\end{array}
\right. \,\,\, \textnormal{ and } \,\,\, 
S_{\varepsilon}(s) = \left\{ 
\begin{array}{l}
\sin(s), \, if \varepsilon = 1, \\
\sinh(s), \, if \varepsilon = -1.
\end{array}
\right.
$$
Let $f: M^n := M^n \times I \rightarrow Q^n(\varepsilon) \times \mathbb{R}$ be a hypersurface defined by
\begin{equation}
f(x,s) = g_s(x) + a(s) \partial_{n+2}, \label{function-tojeiro}
\end{equation}
for some smooth function $a : I \rightarrow \mathbb{R}$ with positive derivative on a open interval $I \subset \mathbb{R}$. In this context, the following theorem provides the mentioned classification:

\begin{theorem}[\cite{tojeiro}]
The map $f$ defines, at regular points, a hypersurface that has $T$ as
a principal direction. Conversely, any hypersurface $M^n$ in $Q^n(\varepsilon) \times \mathbb{R}$, $n \geq 2$, with nowhere vanishing angle function that has $T$ as a principal direction is locally given in this way.
\label{thm-tojeiro}
\end{theorem}

Theorem \ref{criterion} is a powerful tool to classify hypersurfaces in $Q^n(\varepsilon) \times \R$. It is used, for example, to classify the totally umbilical hypersurfaces as hypersurfaces locally isometric to rotation hypersurfaces \cite{Veken1, Veken2}. Also, in the same papers, such criterium is used to characterize a class of semi-parallel hypersurfaces $Q^n(\varepsilon) \times \R$. Namely, the authors have shown that every semi-parallel, not totally umbilical hypersurface, with $n \geq 3$ and two distinct principal curvatures, is locally isometric to a rotation hypersurface. This is done by using Theorem \ref{criterion} and the following Lemma:

\begin{lemma}[\cite{Veken2, Veken1}] \label{lema-semi}
Let $M^n$ be a semi-parallel hypersurface of $Q^n(\varepsilon) \times \mathbb{R}$. Let $T$ and $\theta$ be as above. Then there exists a local orthonormal frame field $\left\{ f_1, \ldots, f_n \right\}$ on $M^n$ with
respect to which the shape operator takes one of the following forms:
\begin{enumerate}[i)]
\item $S = \lambda Id$;
\item $S=\left(\begin{array}{cccc}
        \lambda & \quad  & \quad  & \quad\\
        \quad   & \mu    & \quad  & \quad\\
        \quad   & \quad  & \ddots & \quad\\
        \quad   & \quad  & \quad  & \mu
        \end{array}\right),$ \\
        with $ \lambda \mu = - \varepsilon \cos^2 \theta$ and if $n \geq 3$, then $T = ||T|| f_1$;
\item $S=\left(\begin{array}{cccccc}
        0 & \quad  & \quad  & \quad & \quad & \quad \\
        \quad   & \mu_1    & \quad  & \quad & \quad & \quad \\
        \quad   & \quad  & \mu_1 & \quad & \quad & \quad \\
        \quad   & \quad  & \quad  & \ddots & \quad & \quad \\
        \quad   & \quad  & \quad  & \quad & \mu_2 & \quad \\
        \quad   & \quad  & \quad  & \quad & \quad & \mu_2 
        \end{array}\right)$ with $\lambda \mu = - \varepsilon 1$ and $f_1=T= \partial_{n+2}$.   
\end{enumerate}
\end{lemma} 

Consequently, the semi-parallel hypersufaces in $Q^n(\varepsilon) \times $, $n>3$ are given by one of the following classes: the umbilical hypersurfaces, an open part of the rotation hypersurface with profile curve determined by the equation $\lambda \mu + \varepsilon \cos^2 \theta = 0$ or an open part of the hypersurface $\overline{M}^{n-1} \times \R$, where $\overline{M}^{n-1}$ is a semi-parallel hypersurface in $Q^n(\varepsilon)$ (see \cite[Theorem 5]{Veken1} and \cite[Theorem 4.2]{Veken2}).

\section{Proof of the main results}

\subsection{conformally flat hypersurfaces}

It is a well known fact that $\mathbb{S}^n \times \mathbb{R}$ and $\mathbb{H}^n \times \mathbb{R}$ are conformally flat Riemannian manifolds (see \cite{lafontaine} for a proof). Therefore, the following theorem, due to Nishikawa and Maeda \cite{Nishikawa}, will be useful to understand the geometry behind the conformally flat hypersurfaces in such espaces:

\begin{theorem}[\cite{Nishikawa}] \label{nishikawa-maeda} Let $M^n$ be a hypersurface of a conformally flat Riemannian
manifold $\tilde{M}^{n+1}$ $n>3$. Then $M^n$ is conformally flat if and only if at each point of $M^n$, the shape operator $S$ of $M^n$ is one of the following types:
\begin{enumerate}[i)]
\item $S=\lambda I$, where $I$ the identity transformation,
\item $S$ has two distinct eigenvalues of multiplicity $n-1$ and $1$.
respectively. 
\end{enumerate}
\end{theorem}

We are now in position to proof our first main result:


\begin{proof}[Proof of Theorem \ref{conformally}]
The first part is a direct application of Theorem \ref{nishikawa-maeda}. Since a rotation hypersurface in $Q(\varepsilon) \times \mathbb{R}$ agrees with \emph{i)} or \emph{ii)}, it is a conformally flat hypersuface.

Conversely, let $M^n$ be a conformally flat hypersurface in $Q^n(\varepsilon) \times \mathbb{R}$. By Theorem \ref{nishikawa-maeda}, the shape operator $S$ associated to $M^n$ either has the form $S = \lambda I$, where $I$ is the identity transformation or it has two distinct eingenvalues $\mu$ and $\lambda$ of multiplicity, $n-1$ and $1$. Let us consider each case separately:  

\emph{i)} If $S = \lambda I$, $M^n$ is a totally umbilical hypersurface. Since it is not totally geodesic, it follows by the classification of totally umbilical hypersufaces in $\mathbb{S}^n \times \mathbb{R}$ (see \cite{Veken1}, Theorem 4) and $\mathbb{H}^n \times \mathbb{R}$ (see \cite{Veken2}, Theorem 3.3)  that $M^n$ is a rotation hypersurface.

\emph{ii)} Since $T$ is an eigenvector of the shape operator, we follow \cite{tojeiro} to write the hypersurface $M^n$ locally as given in (\ref{function-tojeiro}). The unit normal of $f$ is given by 
$$
\eta (x,s) = -\dfrac{a'(s)}{\sqrt{1+a'(s)^2}} N_s(x) + \dfrac{1}{\sqrt{1+a'(s)^2}} \partial_{n+2},
$$
where $N_s(x) = - \varepsilon S_{\varepsilon}(s)g(x)+C_{\varepsilon}(s) N(x)$ is the unit normal of $g_s$. Therefore, the principal curvatures of $M^n$ are given by
\begin{equation}
\begin{array}{rcl}
k_i^f &=& - \dfrac{a'(s)}{\sqrt{1+a'(s)^2}} k_i^s, \, 1 \leq i \leq n-1, \\
k_n &=& -\dfrac{a''(s)}{(\sqrt{1+a'(s)^2})^3},
\end{array}
\end{equation}
where $k_i^s$ are the principal curvatures of $g_s$ and $ST = k_n T$.
If the shape operator associated to $M^n$ has two distinct eingenvalues, namely $\lambda$ and $\mu$, with multiplicity $1$ and $n-1$, the only possibility we have is 
$$
\mu = - \dfrac{a'(s)}{\sqrt{1+a'(s)^2}} k^s, \,\,\, \lambda = -\dfrac{a''(s)}{(\sqrt{1+a'(s)^2})^3},
$$
with $k^s = k_i^s$, for all $i$. Consequently, the shape operator of $M^n$ has the form as given in Theorem \ref{criterion}. Since $\lambda=\lambda(\mu, \theta)$, the proof is completed by showing that the angle function $\theta$ does not vary in directions orthogonal to $T$, as stated in Remark \ref{rem-criterion}. This is done by considering equation (\ref{Xcos}). Since $ST = \lambda T$, one has $X[\cos \theta ] = - \langle X, ST \rangle = 0$, which concludes the proof.

\end{proof}

\begin{remark}
It is important to note that the totally geodesic hypersurfaces in $Q^n(\varepsilon) \times \mathbb{R}$ are completely classified. They are given as an open part of $N^{n-1}(\varepsilon) \times \mathbb{R}$. with $N^{n-1}(\varepsilon)$ a totally geodesic hypersurface of $Q^n(\varepsilon)$, or an open part of $Q^n(\varepsilon) \times \left\{ t_0 \right\}$, for $t_0 \in \mathbb{R}$ (see these results in \cite{Veken1} and \cite{Veken2}). 
In this case, the totally geodesic hypersurface will be a rotation hypersurface only when $M^n = Q^{n-1}(\varepsilon) \times \mathbb{R}$.
\end{remark}

\subsection{Radially flat hypersurfaces}

\begin{proof}[proof of Theorem \ref{radially}] The proof consists of applying the Gauss equation (\ref{gauss}) to calculate $\langle R(e_i, T) e_i, T \rangle$, when $e_i \neq T$, $ 2 \leq i \leq n$, and $T$ are principal directions of $M^n$, such that $ST = \lambda T$ and $Se_i = \mu_i e_i$. In this case, one has:
\begin{equation}
\langle R(e_i, T) T, e_i \rangle =  ||T||^2 (\mu_i \lambda + \varepsilon \cos^2 \theta), \label{semi-gauss}
\end{equation}

If $M^n$ is a radially flat hypersurface  and  $ST=\lambda T$, with $\lambda \neq 0$, the equation (\ref{semi-gauss}) implies that $ \lambda \mu_i = -  \varepsilon \cos^2 \theta$.
Let $e_1 = \dfrac{T}{||T||}$ and write $Se_1 = \mu_1 e_1$. Then, $\mu_1 = \lambda$ and $\mu_i = \mu -  \dfrac{\varepsilon \cos^2 \theta}{\lambda}$, $ 2 \leq i \leq n$. It follows by theorem \ref{criterion} and Remark \ref{rem-criterion} that $M^n$ is a rotation hypersurface. For all indices $i, \, j, \, k, \, l$, the equation (\ref{eq-semi-parallel}), evaluated in the principal directions is given by
\begin{equation}
\begin{array}{rcl}
(R \cdot h) (e_i, e_j, e_k, e_l) &=& - h(R(e_i,e_j)e_k,e_l) - h(R(e_i,e_j)e_l,e_k)  \\
&=& - (\mu_l - \mu_k) \left[ (\varepsilon + \mu_i \mu_j) R^0_{ijkl} + \right. \\
&& + \left. \varepsilon ||T||^2 (\delta_{k1} R^0_{ijl1} + \delta_{l1}R^0_{ij1k}) \right],
\end{array} \label{test-semi-parallel}
\end{equation}
where $R^0_{abcd} := \delta_{ad} \delta_{bc} - \delta_{ac} \delta_{bd}$. We claim that the left-hand side of equation (\ref{test-semi-parallel}) vanishes for all indices. If $k \neq 1$ and $l \neq 1$, then $\mu_k = \mu_l$ and consequently $R \cdot h = 0$. When $k=1$ and $l \neq 1$, the equation (\ref{test-semi-parallel}) reduces to
$$
(R \cdot h) (e_i, e_j, e_1, e_l)=(\varepsilon \cos^2 \theta + \mu_i \mu_j ) R^0_{ijl1}.
$$
If $i \neq 1$ and $j \neq 1$ or $i=j=1$, we have $R^0_{ijl1}=0$ and the statement is proved. When $i=1$, one has
$$
(R \cdot h) (e_1, e_j, e_1, e_l)=(\varepsilon \cos^2 \theta + \lambda \mu_j ) R^0_{1jl1}.
$$
Since $\lambda \mu_j = -  \varepsilon \cos^2 \theta$, the affirmation is also true in this case. The remaining cases are treated in a completely analogous way. Therefore $(R \cdot h) (e_i, e_j, e_k, e_l) = 0$ for all indices $i, \, j, \, k, \, l$ and we conclude, by linearity, that $R \cdot h \equiv 0$, consequently $M^n$ is semi-parallel.

Conversely, if $M^n$ is a semi-parallel, not totally umbilical hypersurface, it follows from Lemma \ref{lema-semi} that the shape operator of $M^n$ takes the form \emph{ii)} or \emph{iii)}. In any case, one has $ST = \lambda T$. By applying the equation (\ref{semi-gauss}) again, one has $\langle R(e_i, T) T, e_i \rangle=0$. In fact, let us consider each case separately:
\begin{enumerate}[a)]
\item If $S$ takes the form as given in \emph{ii)}, then $Se_i = \mu ei$ and $\mu \lambda = - \varepsilon \cos^2 \theta$. Therefore $\langle R(e_i, T) e_i, T \rangle = ||T||^2 (\mu \lambda + \varepsilon \cos^2 \theta)=0$. 
\item If $S$ takes the form as given in \emph{iii)}, then $\lambda=0$ and $T=\partial_{n+2}$, which is equivalent to $\cos \theta =0$. Consequently $\langle R(e_i, T) T, e_i \rangle=0$.
\end{enumerate}
Since $\langle R(e_i, T) T, e_j \rangle=0$, for $i \neq j$, we conclude, by linearity, that $M^n$ is radially flat.
\end{proof}

\begin{remark}
When $M^n$ is radially flat and $T$ is a principal direction, with principal curvature $\lambda = 0$, it follows by Gauss equation (\ref{gauss}) that $\cos \theta = 0$ and therefore $M^n = \overline{M}^{n-1} \times \mathbb{R}$, where $\overline{M}^{n-1}$ is a hypersurface of $Q^n(\varepsilon)$. It is no longer true, in general, that $M^n$ in this case is semi-parallel. In fact, when $M^n$ takes this form, it will be semi-parallel if, and only if, $\overline{M}^{n-1} \subset Q^n(\varepsilon)$ is semi-parallel (see \cite[Theorem 5]{Veken1} and \cite[Theorem 4.2]{Veken2}).

On the other hand, when $M^n$ is a semi-parallel, totally umbilical hypersurface in $Q^n(\varepsilon) \times \mathbb{R}$, it does not follow directly that $M^n$ is radially flat. In fact, $M^n$ will be radially flat when:
\begin{enumerate}[a)]
\item  $M^n$ is an open part of the the totally geodesic $\mathbb{S}^{n-1} \times \mathbb{R}$. In fact, we must have the shape operator $S \equiv 0$ and $\cos \theta \equiv 0$.
\item $M^n$ is a hypersurface in $\mathbb{H}^n \times \mathbb{R}$ with $\lambda^2 = \cos^2 \theta$. Particularly, if $\lambda \equiv 0$, then $M^n$ is is an open part of a totally geodesic $\mathbb{M}^{n-1} \times \mathbb{R}$, where $\mathbb{M}^{n-1} \subset \mathbb{H}^n$ is a totally geodesic hypersurface. 
\end{enumerate}

\end{remark}

\section{Applications}

In this section, we will present some applications as corollaries of Theorem 1 (corollaries 11, 12 and 13) and Theorem 2 (corollary 14).

Firstly, we use Theorem \ref{conformally} to classify the rotation hypersurfaces in $Q^n(\varepsilon) \times \mathbb{R}$ that are Einstein manifolds:

\begin{cor}
Let $M^n$, $n>3$, be a rotation hypersurface in $Q^n(\varepsilon) \times \mathbb{R}$. If $M^n$ is an Einstein manifold, then $M^n$ has constant sectional curvature. \label{cor-einstein}
\end{cor}

\begin{proof}
The proof follows directly by the fact that $M^n$ is conformally flat, given in the first part of Theorem 1. It is a well known result (see \cite{kuhnel} for details) that, if $(M^n,g)$ is a manifold with constant sectional curvature, then $(M^n, \varphi g)$, where $\varphi >0$ is a smooth function defined on $M^n$, is an Einstein manifold if and only if $(M^n, \varphi g)$ has constant sectional curvature.
\end{proof}

In \cite{manfio}, rotation hypersurfaces in $Q^n(\varepsilon) \times \mathbb{R}$, $n \geq 3,$ with constant sectional curvature were completely classified in terms of the respective profile curve. Therefore, corollary above and this classification give a complete classification of the rotation hypersurfaces in $Q^n(\varepsilon) \times \mathbb{R}$ where the induced metric is Einstein.

A natural generalization of Einstein manifolds are the Ricci solitons. In recent years, this topic has attracted a lot of attention, mainly because their are special solutions to the Ricci flow, which was use to prove the Poincar\'e conjecture (see \cite{cao} and the references therein). A Riemannian manifold $(M,g)$ endowed with a smooth vector field $V$ is a Ricci soliton if
\begin{equation}
\textnormal{Ric} + \dfrac{1}{2} \mathcal{L}_V g  = c g, \label{ricci-soliton}
\end{equation}
where $c$ is a real constant and $\mathcal{L}_V g$ is the Lie derivative of $g$ with respect to $V$. The vector field $V$ is called potential vector field. The Ricci soliton is called shrinking when $c>0$, steady when $c=0$, and expanding when $c<0$.

If $V$ is the gradient of a smooth function $f$, then the Equation (\ref{ricci-soliton}) takes the form
\begin{equation}
 \textnormal{Ric} + \textnormal{Hess}_f = c g, \label{ricci-soliton-gradient}
\end{equation}
where $\textnormal{Hess}_f$ is the Hessian of $f$. In this case, the $(M,g)$ is called gradient Ricci soliton and the function $f$ is called potential function.

Ricci solitons are also considered as Riemannian submanifolds (see \cite{chen} for a survey on the topic). Although the extrinsic geometry Ricci solitons given by conformally flat hypersurfaces in $\R^n$ and $S^n$ is well described (\cite{cho}), their behavior on conformally flat hypersurfaces in $Q^n(\varepsilon) \times \mathbb{R}$ is still unknown.  In this context, Theorem \ref{conformally} is used to provide a class of  conformally flat hypersurfaces $Q^n(\varepsilon) \times \mathbb{R}$, with the structure of Ricci solitons, whose potential vector field is $T$. This class is given by the following corollary:

\begin{cor}
Let $M^n$, $n>3$, be a conformally flat hypersurface in $Q^n(\varepsilon) \times \mathbb{R}$, not totally umbilical. Suppose that $M^n$ is a Ricci soliton, whose potential vector field is $T$. If $T$ is an eigenvector of the shape operator, associated to an eingenvalue of multiplicity $1$, then $M^n$ is rotational. \label{cor-soliton}
\end{cor}

\begin{proof}
Since $M^n$ is a conformally flat hypersurface, not totally umbilical, Theorem \ref{conformally} tells that the shape operator of $M^n$ has two distinct eingenvalues $\lambda$ and $\mu$, of multiplicity $1$ and $n-1$, respectively. Let $e_i$, $1 \leq i \leq n-1$, eigenvectors associated to $\mu$. Since $ST=\lambda T$, it follows by Gauss Equation (\ref{gauss}) that 
\begin{equation}
\textnormal{Ric}(e_i, e_i) = (n-2) (\mu^2 + \varepsilon) + \varepsilon \cos^2 \theta + \lambda \mu. \label{ricci-ei}
\end{equation}
On the other hand, it follows by the first Equation of (\ref{Xcos}) that $\nabla_{e_i} T = \mu \cos \theta e_i$. Consequently 
\begin{equation}
\mathcal{L}_T (e_i, e_i) = 2 \mu \cos \theta. \label{lie-ei}
\end{equation}
By considering (\ref{ricci-ei}) and (\ref{lie-ei}) in Equation (\ref{ricci-soliton}), with potential vector field is $T$, one has
\begin{equation}
\mu \cos \theta + (n-2) (\mu^2 + \varepsilon) + \varepsilon \cos^2 \theta + \lambda \mu = c. \label{ricci-soliton-ei}
\end{equation}
Equation (\ref{ricci-soliton-ei}) tells that $\lambda = \lambda(\mu, \theta)$. By Theorem \ref{conformally}, one has that $M^n$ is a rotation hypersurface.
\end{proof}

One of the main argument in the proof of Corollary \ref{cor-soliton} is the Equation (\ref{ricci-soliton-ei}), which shows the functional dependence between $\lambda$, $\mu$ and $\theta$. Therefore, intrinsic properties that imply in such dependence can also be used to show that a conformally flat hypersurface is rotational. In the next corollary, we consider conformally flat hypersurfaces with constant scalar curvature.

\begin{cor}
Let $M^n$, $n>3$, be a conformally flat hypersurface in $Q^n(\varepsilon) \times \mathbb{R}$, not totally umbilical, with constant scalar curvature. If $T$ is an eigenvector of the shape operator, associated to an eingenvalue of multiplicity $1$, then $M^n$ is rotational. \label{cor-scalar}
\end{cor}

\begin{proof}
The proof is similar to Corollary \ref{cor-soliton}. If $M^n$ has two distinct eingenvalues $\lambda$ and $\mu$, of multiplicity $1$ and $n-1$, respectively, with $ST=\lambda T$, a straightforward computation shows that the scalar curvature is given by
\begin{equation}
\rho = (n-1)(n-2)(\mu^2 + \varepsilon) + 2(n-1)(\lambda \mu + \varepsilon cos^2 \theta). \label{scalar-curvature}
\end{equation}
If $\rho$ is constant, one has $\lambda = \lambda (\theta, \mu)$. It follows by Theorem \ref{conformally} that $M^n$ is a rotation hypersurface.
\end{proof}

Let us observe that the vector field $T$ is actually a gradient vector field. In fact, if we express a point $p \in M^n$ as $p=(\varphi, h) \in  Q^n(\varepsilon) \times \R$, then $T$ is the gradient of the height function $h$. In this way, the Ricci soliton given in Corollary \ref{cor-soliton} is a gradient Ricci soliton. 

A gradient Ricci soliton is rigid if it is isometric to a quotient $ N \times_{\Gamma} \R^k$ where $N$ is an Einstein manifold, $f = \frac{c}{2}|x|^2$ on the Euclidean factor and $\Gamma$ acts freely on $N$ and by orthogonal transformations on $\R^k$ (\cite{petersen2, petersen1}). In \cite[Theorem 1.2]{petersen1}, Petersen and Wylie proved that a a gradient Ricci soliton $\textnormal{Ric} + \textnormal{Hess}_f = cg$  is rigid if, and only if, it has constant scalar curvature and the sectional curvatures $K(X, \nabla f)=0$, for any vector field.

As a consequence of Theorem 2, we obtain when a hypersurface in $Q^n(\varepsilon) \times \R$ is a rigid gradient Ricci soliton:

\begin{cor}
Let $M^n$, $n>3$, be a Ricci soliton in $Q^n(\varepsilon) \times \R$ with potential vector field $T$ and constant scalar curvature. If $M^n$ is a rigid gradient Ricci soliton, and $T$ is a principal direction for a principal curvature $\lambda \neq 0$, then $M^n$ is a semi-parallel hypersurface. Conversely,
If $M^n$ is a semi-parallel, not totally umbilical hypersurface, then $M^n$ is a rigid gradient Ricci soliton.
\end{cor}
\begin{proof}
Since $T$ is a gradient vector field, $M^n$ is a gradient Ricci soliton. By \cite[Theorem 1.2]{petersen1} $M^n$ is rigid if, and only if, it has constant scalar curvature and is radially flat. Consequently, the proof follows directly by Theorem 2. 
\end{proof}

%
%
%
%
%
%
%
%
%
%

\bibliographystyle{amsplain}

\end{document}